\documentclass[14pt]{amsart}
\usepackage{latexsym, amssymb}
\usepackage[usenames]{color}
\usepackage[all]{xy}

\newcommand{\be}{\begin{equation}}
\newcommand{\ee}{\end{equation}}
\newcommand{\beq}{\begin{eqnarray}}
\newcommand{\eeq}{\end{eqnarray}}

\newtheorem{thm}{Theorem} [section]
\newtheorem{cor}[thm]{Corollary}
\newtheorem{lem}[thm]{Lemma}

\newtheorem{prop}[thm]{Proposition}
\theoremstyle{definition}
\newtheorem{defn}[thm]{Definition}
\theoremstyle{remark}

\numberwithin{equation}{section}
\begin{document}
\title[ Newton polygons for $L$-functions of generalized Kloosterman sums]
{  Newton polygons for $L$-functions of generalized Kloosterman sums}
\begin{abstract}
%We consider a kind of generalized Kloosterman family of nondegenerate toric exponential sums over the torus.
%To study the associated $L$-functions, we construct acyclic relative $p$-adic cohomology by following Haessig and Sperber's
% construction and derive the explicit form of bases of the top dimensional spaces.
%Using these bases, we give a method to compute lower bounds of Newton polygons of $L$-functions.
%After the change of bases, we obtain the Dwork deformation equations and prove that there are strong Frobenius structures on deformation equations.
%Our work adds some new evidences for Dwork's conjecture.

%In a former paper, Haessig and Sperber constructed a relative $p$-adic cohomolgy which is acyclic except the top dimension for a family of generalized Kloosterman sums.
%special case of Haessig and Sperber's family.

In this paper, we study the Newton polygons for the $L$-functions of $n$-variable generalized Kloosterman sums.
Generally, the Newton polygon has a topological lower bound, called the Hodge polygon.
In order to determine the Hodge polygon, we explicitly construct a basis of the top dimensional Dwork cohomology.
 Using Wan's decomposition theorem and diagonal local theory, we obtain when the Newton polygon coincides with the Hodge polygon. In particular, we concretely get the slope sequence for $L$-function of $\bar{F}(\bar{\lambda},x):=\sum_{i=1}^nx_i^{a_i}+\bar{\lambda}\prod_{i=1}^nx_i^{-1}$.
  %with \textcolor{blue}{ $\{a_i\}_{i=1}^n$ being pairwise coprime if $n\ge 2$.}

\end{abstract}

\author[C.L. Wang]{Chunlin Wang}
\address{School of Mathematical Sciences, Sichuan Normal University, Chengdu 610064, P.R. China}
\email{c-l.wang@outlook.com}

\author[L.P. Yang]{Liping Yang$^*$}
\address{School of Mathematical Sciences, Capital Normal University, Beijing 100048, P.R. China}
\email{yanglp2013@126.com}

%\date{\today}
\thanks{$^*$ L.P. Yang is the corresponding author and supported by Beijing Postdoctoral Research Foundation (No. 202022073).
C.L. Wang is supported by National Natural Science Foundation of China (No. 11901415) and the China Scholarship Council (No. 201808515110).
}
\maketitle

\section{Introduction}

Let $p>2$ be a prime, and $\mathbb{F}_q$ be the finite field of $q$ elements with characteristic $p$.
For each positive integer $k$, let $\mathbb{F}_{q^k}$ be the finite extension of $\mathbb{F}_q$ of degree $k$.
Let $\mathbb{Q}_p$ be the $p$-adic number field, and $\mathbb{Z}_p$ be the ring of $p$-adic integers.
For any Laurent polynomial $f(x)\in \mathbb{F}_q[x_1^{\pm},\cdots ,x_n^{\pm}]$, the toric exponential sum is defined as
$$S_k(f):=\sum_{x\in (\mathbb{F}_{q^{k}}^{*})^{n}} \zeta_p^{Tr_k f(x)},$$
where $Tr_k: \mathbb{F}_{q^k}\rightarrow \mathbb{F}_p$ is the trace of the field extension and $\mathbb{F}_{q^{k}}^{*}$ denotes the set of non-zero elements in $\mathbb{F}_{q^{k}}$.
By a theorem of Dwork--Bombieri--Grothendieck, the following exponential generating $L$-function is a rational function
$$L(f,T):=\exp\Big(\sum_{k=1}^{\infty} \frac{S_k(f)T^k}{k}\Big)=\frac{\prod_{i=1}^{d_1}(1-\alpha_i T)}{\prod_{j=1}^{d_2}(1-\beta_j T)}, $$
where $\alpha_{i}(1\le i\le d_1)$ and $\beta_j(1\le j\le d_2)$ are non-zero algebraic integers.
From Deligne's integrality theorem, one has the following estimate
$$|\alpha_i|_p=q^{-r_i}, |\beta_j|_p=q^{-s_j},r_i\in \mathbb{Q}\cap[0,n], s_j\in \mathbb{Q}\cap[0,n]$$
with normalized $p$-adic absolute value $|\cdot|_p$ such that $|q|_p=q^{-1}$.
The rational number $r_i$ (resp. $s_j$)
is called the {\it slope} of $\alpha_i$ (resp. $\beta_j$) with respect to $q$.
The $p$-adic Riemann hypothesis for the $L$-function $L(f,T)$ is to determine the slopes of the zeros and poles.
This is an extremely hard problem if there is no smoothness condition on $f$.

Write $f(x)=\sum \bar{a}_vx^v\in\mathbb{F}_q[x_1^{\pm},...,x_n^{\pm}]$. The support of $f$, denoted by ${\it Supp} (f)$,
is defined as
$$ Supp (f):=\{v: \bar{a}_v\neq 0\}.$$
Let $\Delta (f)$ be the convex closure of $Supp(f)$ and the origin.
%and let $\Delta_{\infty} (f)$ be the convex closure of $\Delta (f)\cup \{0\}$.
If $\delta$ is a subset of $\Delta (f)$, we define the restriction of $f$ to $\delta$
to be $$f^{\delta}=\sum_{v\in \delta}\bar{a}_vx^v.$$

\begin{defn}
The Laurent polynomial $f$ is called {\it nondegenerate} if for each closed face $\delta$ of $\Delta (f)$
of arbitrary dimension which does not contain the origin, the $n$ partial derivatives
$$\Big\{\frac{\partial f^{\delta}}{\partial x_1},...,\frac{\partial f^{\delta}}{\partial x_n} \Big\}$$
have no common zeros with $x_1\cdots x_n\neq 0$ over the algebraic closure of $\mathbb{F}_q$.
\end{defn}

Assume that $\Delta(f)$ is of dimension $n$ and $f$ is nondegenerate with respect to $\Delta(f)$. Adolphson and Sperber \cite{AS89} showed that the $L$-function $L(f,T)^{(-1)^{n-1}}$ is a polynomial of degree $n!Vol(\Delta(f))$, where $Vol(\Delta(f))$ denotes the volume of $\Delta(f)$.
%If we write
%$$L(f,T)^{(-1)^{n-1}}=\sum_{i=0}^{n!Vol(\Delta(f))}A_iT^i,$$
%then the $q$-adic {\it Newton polygon} of $L(f,T)^{(-1)^{n-1}}$, denoted by ${\rm NP}(f)$, is the lower convex closure in $\mathbb{R}^2$ of the following points
%$$(r,ord_q A_r),\ r=0,1,...,n!Vol(\Delta(f)).$$
%Further more, the slope sequence $\{ord_q(\alpha_1),\cdots,ord_q(\alpha_{n!Vol(\Delta(f)) }) \}$ is determined by the $q$-adic Newton polygon of $L(f,T)^{(-1)^{n-1}}$.
Then
$$L(f,T)^{(-1)^{n-1}}=\prod_{i=1}^{n!Vol(\Delta(f))}(1-\alpha_iT)=\sum_{i=0}^{n!Vol(\Delta(f))}A_iT^i.$$
The $q$-adic {\it Newton polygon} of $L(f,T)^{(-1)^{n-1}}$, denoted by ${\rm NP}(f)$, is defined to be the lower convex closure in $\mathbb{R}^2$ of the points
$$(r,ord_q A_r),\ r=0,1,...,n!Vol(\Delta(f)).$$
Obviously, determining the slope sequence $\{ord_q(\alpha_1),\cdots,ord_q(\alpha_{n!Vol(\Delta(f)) }) \}$ is equivalent to determining the $q$-adic Newton polygon of $L(f,T)^{(-1)^{n-1}}$.

In general, there is no effectual way to obtain the Newton polygon. A well studied
example is the Kloosterman polynomial
$$ f^{(n)}(\bar{\lambda},x)=x_1+\cdots+x_n+\frac{\bar{\lambda}}{x_1\cdots x_n}$$
with $\bar{\lambda }\in \mathbb{F}_q^{*}$.
 For $n=1$, Dwork \cite{[Dwo74]} derived that $\{0,1\}$ is the slope sequence of the $L$-function $L(f^{(1)}(\bar{\lambda},x), T)$.
Later, Sperber \cite{Sp1},\cite{Sp2} generalized Dwork's result to the $n$-variable case
by showing that the slope sequence of $L(f^{(n)}(\bar{\lambda},x),T)^{(-1)^{n-1}}$
is $\{0,1,\cdots,n\}$. But the methods used by Dwork and Sperber are very complicated.

 Adolphson and Sperber \cite{AS89} proved that the Newton polygon has a topological lower bound called the Hodge polygon, which is easier to calculate. In particular, the polynomial $f$ is called {\it ordinary} if the Newton polygon equals to its Hodge polygon. Hence a general way to calculate the Newton polygon is usually to compute the Hodge polygon first, and then to decide when the Newton polygon and the Hodge polygon coincide. Actually, for the diagonal Laurent polynomial, Wan \cite{WD2} got some criterions to decide when the Newton polygon and the Hodge polygon coincide. For the non-diagonal Laurent polynomial, one main technique is Wan's decomposition theorem, that is,
decomposing $\Delta$ into small pieces which is diagonal and easy to deal with, the related work can be found in \cite{BS}, \cite{CL}, \cite{WD2}, \cite{XZ}, \cite{ZF}. Further decomposition methods for Newton polygons are developed in \cite{Le}.

As a generalization of classical Kloosterman polynomial, the Newton polygon of the following polynomial
\begin{equation}\label{eq4}
\bar{F}(\bar{\lambda},x):=\sum_{i=1}^nx_i^{a_i}+\bar{\lambda}\prod_{i=1}^nx_i^{-d_i}
\end{equation}
 interests many mathematicians.
For the case $a_1=\cdots=a_n=1$,
 Wan \cite{WD2} used decomposition theorems to prove that $\bar{F}(\bar{\lambda},x)$ is ordinary for all $p$ such that $p-1$ is divisible by $lcm(d_1,\cdots,d_n)$, and
this result also was proved by Sperber \cite{Sp3} for large $p$, see \cite{Sp2} for the classical case.
On the other hand, Bellovin et al. \cite{BS} studied the case $d_1=\cdots=d_n=1$
and obtained when $\bar{F}(\bar{\lambda},x)$ is ordinary. But the Hodge polygon is hard to calculate. Using the combinatorial definition of Hodge polygon, Bellovin et al. computed the Hodge polygon under certain numeric restriction.
Recently, the authors of the present paper studied the slope sequence of the $L$-function $L(\bar{F}(\bar{\lambda},x),T)^{(-1)^{n-1}}$ for $n=2$ and obtained the explicit slope sequence for the case when $d_1=d_2=1$ and $a_1, a_2$ being coprime \cite{WY}.

In this paper, we will require that $a_i, d_i$ are positive integers not divisible by $p$ for $i=1,\cdots,n$.
For convenience, we also use $\bar{F}$ instead of $\bar{F}(\bar{\lambda},x)$ sometimes.
Our main work is to use the above idea to study the Newton polygon of $L(\bar{F},T)^{(-1)^{n-1}}$. That is, we compute the Hodge polygon first, then decide when $\bar{F}$ is ordinary, and at last we can derive the Newton polygons under certain conditions.

Instead of the combinatorial definition, we will use the cohomology definition to compute the Hodge polygon.
According to Adolphson and Sperber \cite{AS89}, the $L$-function $L(\bar{F}(\bar{\lambda},x),T)^{(-1)^{n-1}}$ can be expressed as the characteristic polynomial of a Frobenius map acting on the top dimensional cohomology space of a certain Koszul complex $\Omega^{\bullet}(\mathcal{C}_0,\nabla(D))$ (constructed as in (\ref{eq3} )). The Hodge polygon is then totally determined by a basis of the top dimensional cohomology space $H^n(\Omega^{\bullet}(\mathcal{C}_0,\nabla(D)))$.
%Adolphson and Sperber \cite{AS89} proved that the Hodge polygon can be totally determined by weight of a basis of the top %cohomology space $H^n(\Omega^{\bullet}(\mathcal{C}_0,\nabla(D)))$,
This allows us to study the explicit form of the basis.

In order to
describe our results, we introduce some notations.
Here $\Delta(\bar{F})$ denotes the convex closure of $Supp(\bar{F})$ and the origin.
For $v\in \mathbb{Z}^n$, let $w(v)$ be the least nonnegative rational number such that $v\in w(v)\Delta(\bar{F})$. For every set $A$,
 let $|A|$ denote the cardinality of $A$.
Define set
$\mathcal{B}$  as follows. If $n=1$, then let $\mathcal{B}:=\{v\in \mathbb{Z}: -d_1<v\le a_1\} $.
If $n\ge 2$, then define $\mathcal{B}$ to be the set of $(v_1,\cdots, v_n)\in \mathbb{Z}^n$ such that
$-d_i<v_i\le a_i$ for all $i=1,\cdots,n$, and
\begin{align}\label{eq1}
 \frac{d_j}{d_i}(v_i-a_i)\le v_j< \frac{d_j}{d_i}v_i+a_j.
\end{align}
for integers $i,j$ with $1\le i<j\le n$.
Let $e^{*}=lcm(a_1,d_1)$ for $n=1$, and
$$e^{*}=lcm(a_1,...,a_n)\cdot lcm(d_1,...,d_n)$$ for $n\ge 2$.

\begin{thm}\label{thm10}  Let $\hat{B}:=\{\tilde{\pi}^{e^{*}w(u)}x^{v}: v\in \mathcal{B}\}.$
Then $\hat{B}$ forms a basis of $H^n(\Omega^{\bullet}(\mathcal{C}_0,\nabla(D)))$.
\end{thm}

To prove Theorem \ref{thm10}, we first show that $\hat{B}$ contains a basis of $H^n(\Omega^{\bullet}(\mathcal{C}_0,\nabla(D)))$ and then show the number of lattice points $|\mathcal{B}|$ equals to the rank of $H^n(\Omega^{\bullet}(\mathcal{C}_0,\nabla(D)))$. This is the technical part of this paper.
 As a direct consequence of Theorem \ref{thm10}, we have the following result for the Hodge polygon.

\begin{cor}\label{thm2}
Let $k_i:=|\{v\in \mathcal{B}: w(v)=i/e^{*}\}|$ for $0\le i\le ne^{*}$. Then the Hodge polygon $HP(\Delta(\bar{F}))$ is the convex closure of $(0,0)$ and
$$\left(\sum_{i=0}^{m}k_i, \sum_{i=0}^{m}\frac{ik_i}{e^{*}}\right),\quad  m=0,1,...,ne^{*}. $$
\end{cor}

It follows from Corollary \ref{thm2} that the Hodge polygon can be computed concretely.
Using Wan's decomposition theorem and diagonal local theory, we can decide when the Newton polygon of
$L(\bar{F}(\bar{\lambda},x),T)^{(-1)^{n-1}}$ coincides with the Hodge polygon. Hence we have
\begin{thm}\label{thm9}
If $p\equiv 1\mod e^{*}$, then $NP(\bar{F})$ equals to the Hodge polygon $HP(\Delta(\bar{F}))$.
\end{thm}

In particular, we have
\begin{thm}\label{thm1}
For $n\ge 1$, let $d_i=1$ for $1\le i\le n$, $\gcd(a_i,a_j)=1\ {\rm for\ any}\ 1\le i< j\le n$ and $p\equiv 1 \mod \prod_{i=1}^na_i$.
Then each slope of $L(\bar{F}(\bar{\lambda},x),T)^{(-1)^{n-1}}$ is with multiplicity one and the slope sequence equals to
$$\Big\{\sum_{i=1}^n\frac{u_i}{a_i}: u_i= 0,\cdots, a_i\Big\}$$
as a set.
\end{thm}

This paper is organized as follows. In section 2, we review Dwork's cohomology theory.
In section 3, we first prove Theorem \ref{thm10}, from which Corollary \ref{thm2} follows. Then by reviewing Wan's decomposition theorem and diagonal local theory we prove Theorem \ref{thm9} and Theorem \ref{thm1}.

\section{Dwork cohomology}

In this section, we give a brief review on $p$-adic cohomology theory of Dwork type, for more details, see \cite{AS89} or \cite{HS17}.
Let $\bar{F}(\bar{\lambda},x)$ be given by (1.1).
Recall that $\Delta(\bar{F})$ is the convex closure of $Supp(\bar{F})$ and the origin.
%Denote by $A_i=(0,...,0,a_i,0...,0)$ and $\mu=(-d_1,...,-d_n)$. Then
%$$\bar{F}(\bar{\lambda},x)=\sum_{i=1}^n x^{A_i}+\bar{\lambda}x^{\mu}$$
%and $\Delta(\bar{F})$ be the convex closure of $\{\mu, A_1,...,A_n, (0,...,0)\}$.
Note that $p\nmid a_id_i$ for $i\in\{1,\cdots,n\}$.
Then $\bar{F}$ is nondegenerate with respect to $\Delta(\bar{F})$.
%Let $Cone (\bar{F})$ be the union of rays from 0 passing through $\Delta(\bar{F})$.
%Set $M(\bar{F}):=Cone (\bar{F})\cap \mathbb{Z}^n$.
%In this paper, we focus on
%$\bar{f}(x)=\sum_{i=1}^nx_i^{a_i}$ and
%$\mu=\prod_{i=1}^nx_i^{-d_i}$.
%Then $\dim \Delta_{\infty} (\bar{f})=\dim Cone (\bar{f})=n$.
%The points of ${\rm Supp} (\bar{f})$ determine the hyperplane
%$$l_{\sigma}(v):=\Big\{(v_1,v_2,...,v_n) | \sum_{i=1}^n\frac{v_i}{a_i}=1\Big\}.$$
%Clearly, $\mu\notin M(\bar{f})$ and $l_{\sigma}(\mu)<1$.
%Note that $p\nmid \prod_{i=1}^n a_id_i$ and $\bar{f}$ is nondegenerate with respect to $\Delta_{\infty}(\bar{f})$. It follows from Theorem 2.1 of \cite{HS17}
%that for each $\bar{\lambda}\in \mathbb{F}_q$,
%$\bar{F}(\bar{\lambda},x)$ is nondegenerate with respect to $\Delta_{\infty}(\bar{f},\mu)$,
%where $\Delta_{\infty}(\bar{f},\mu)$ is the convex closure of $\Delta_{\infty}(\bar{f})\cup \{\mu\}$.
Let $Cone(\bar{F})$ be the cone in $\mathbb{R}^n$ over $\Delta(\bar{F})$ and let
$M(\bar{F}):=Cone(\bar{F})\cap \mathbb{Z}^n$. Notably $M(\bar{F})=\mathbb{Z}^n$.
 %If $\tau$ is a closed face of $\Delta(\bar{F})$ not containing 0, then we say $\tau$ is {\it a face at $\infty$}.
For a point $v\in \mathbb{R}^n$, the {\it weight} $w(v)$ is defined to be
the smallest nonnegative real number $c$ such that $v\in c\cdot \Delta(\bar{F})$. If there is no such $c$, we then define
$w(v)=\infty$. There is a positive integer $e$ such that $w(\mathbb{Z}^n)\subset (1/e)\mathbb{Z}_{\ge 0}$. The weight function $w(v)$ has the following properties:

(a) $w(v)=\inf\big\{\sum_{u\in Supp(\bar{F})} \alpha_u : \sum_{u\in Supp(\bar{F})} \alpha_uu=v,\alpha_u\in\mathbb{Q}, \alpha_u\ge 0\big\}$;

(b) $w(lv)=lw(v)$ for $l\in\mathbb{Z}_{> 0}$;

(c) $w(v+v')\le w(v)+w(v')$, with equality holding if and only if $v$ and $v'$ are cofacial, i.e., $v/w(v)$ and $v'/w(v')$ lie on the same closed face of $\Delta(\bar{F})$.

%Now we express the specific form of the weight function $w$ on the closed subcones of $Cone(\bar{f},\mu)$ corresponding to the codimension one faces
%$\omega$ of $\Delta_{\infty}(\bar{f},\mu)$ at $\infty$.
%If $\omega=\sigma=\Delta(\bar{f})$ and $v\in %Cone(\bar{f})\cap \mathbb{Z}^n=
%M(\bar{f})$,
%then let
%\begin{equation}\label{eq1}
%w(v):=l_{\sigma}(v).
%\end{equation}
%For $\tau\in {\rm Supp}(\bar{f})$, let $C(\tau,\mu)$ denote the segment connecting $\tau$ and $\mu$, and $Cone(\tau,\mu)$ denote the set of rays from origin passing through $C(\tau,\mu)$.
%If $\omega=C(\tau,\mu)$ and $\tau$ is a face of codimension 2 at $\infty$ of $\Delta_{\infty}(\bar{f})$, then
%the affine hyperplane spanned by $C(\tau,\mu)$ has equation $l_{(\tau,\mu)}(v)=1$ where
%\begin{equation}\label{eqn11}
%l_{(\tau,\mu)}(v)=\frac{\phi^{(\tau)}(v)}{\phi^{(\tau)}(\mu)}(1-l_{\sigma}(\mu))+l_{\sigma}(v).
%\end{equation}
%If $v\in Cone(\tau,\mu)\cap \mathbb{Z}^n$, then $w(v):=l_{(\tau,\mu)}(v)$.
Let $R:=\mathbb{F}_q[x_1^{\pm},\cdots ,x_n^{\pm}]$.
One can define an increasing filtration on $R$ by setting, for $i\in \mathbb{Z}_{>0}$,
$$R_{i/e}:=\Big\{\sum_{v}b_vx^v: w(v)\le i/e\ {\rm for\ all}\ v\ {\rm with}\ b_v\neq 0 \Big\}.$$
%Then $R$ has an increasing filtration indexed by $(1/e)\mathbb{Z}_{\ge 0}$.
Let $\bar{R}$ denote the associated graded ring.
Multiplication in $\bar{R}$ obeys the rule
\begin{align}\label{eqn2}
x^u x^{u'}=
&\left\{
  \begin{array}{ll}
x^{u+u'},&  \hbox{if $u$ and $u'$ are cofacial};\\
   0,& \hbox{otherwise}.\\
  \end{array}
\right.
\end{align}

We can construct two complexes as follows.
The spaces in both complexes are the same
\begin{equation}\label{eq0}
\Omega^{i}(\bar{R},\nabla(\bar{F})):=\Omega^{i}(\bar{R},\nabla(\bar{D})):=\bigoplus_{1\le j_1<...<j_i\le n} \bar{R}\frac{dx_{j_1}}{x_{j_1}}\wedge...\wedge \frac{dx_{j_i}}{x_{j_i}}
\end{equation}
with respective boundary operators given by
$$\nabla(\bar{F})\Big( \zeta \frac{dx_{j_1}}{x_{j_1}}\wedge...\wedge \frac{dx_{j_i}}{x_{j_i}}\Big):=\Big(\sum_{l=1}^nx_l\frac{ \partial \bar{F}}{\partial x_l}\zeta \frac{dx_{l}}{x_{l}}\Big)\wedge\frac{dx_{j_1}}{x_{j_1}}\wedge...\wedge \frac{dx_{j_i}}{x_{j_i}}$$
and
$$\nabla(\bar{D})\Big( \zeta \frac{dx_{j_1}}{x_{j_1}}\wedge...\wedge \frac{dx_{j_i}}{x_{j_i}}\Big):=\Big(\sum_{l=1}^n\bar{D}_{l}(\zeta ) \frac{dx_{l}}{x_{l}}\Big)\wedge\frac{dx_{j_1}}{x_{j_1}}\wedge ...\wedge\frac{dx_{j_i}}{x_{j_i}},$$
where $$\bar{D}_l:=x_l\frac{\partial}{\partial x_l}+x_l\frac{ \partial \bar{F}}{\partial x_l}. $$

Since $\bar{F}(\bar{\lambda},x)$ is nondegenerate with respect to $\Delta(\bar{F})$, we have
\begin{thm}[Theorem 2.2, \cite{HS17}]\label{thm0}
The two complexes $\Omega^{\bullet}(\bar{R},\nabla(\bar{F}))$
and $\Omega^{\bullet}(\bar{R}, \nabla(\bar{D}))$
are acyclic except in the top dimension $n$. In both cases, the top dimensional cohomology $H^n$ is a finite free $\mathbb{F}_q$-algebra
of rank $n!Vol(\Delta(\bar{F})$.
For each $i\in (1/e)\mathbb{Z}_{\ge 0}$, we may choose a monomial basis $B_i$
consisting of monomials of weight $i$ for an $\mathbb{F}_q$-vector space $V_i$
such that the $i$-th graded piece $\bar{R}_i$ of $\bar{R}$ may be written as
$$\bar{R}_i=V_i\oplus \sum_{l=1}^nx_l\frac{\partial \bar{F}}{\partial x_l}\bar{R}_{i-1}$$
so that if $B=\cup_{i\in (1/e)\mathbb{Z}_{\ge 0}}B_i$
and $V=\sum_{i\in (1/e)\mathbb{Z}_{\ge 0}}V_i$ is the $\mathbb{F}_q$-vector space with
basis $B$,
then
$$\bar{R}=V\oplus \sum_{l=1}^nx_l \frac{\partial \bar{F}}{\partial x_l}\bar{R} $$
and
 $$\bar{R}=V\oplus \sum_{l=1}^n\bar{D}_l\bar{R}.  $$
\end{thm}

Let
$$H_{\Delta}(i):=\dim_{\mathbb{F}_q}V_i. $$
We now give the definition of the Hodge polygon of $\Delta(\bar{F})$.
%Let $\tilde{D}$ be the least positive integer such that $w(u)\in (1/\tilde{D})\mathbb{Z}_{\ge 0}$ for all $u\in Cone(\tau,\mu)$ \textcolor{red}{(?)}.
\begin{defn}
The {\it Hodge polygon} $HP(\Delta(\bar{F}))$ of $\Delta(\bar{F})$
is the lower convex polygon in $\mathbb{R}^2$ with vertices (0,0) and
$$\Big(\sum_{k=0}^{m}H_{\Delta}(k),\sum_{k=0}^{m}\frac{k}{e}H_{\Delta}(k) \Big), m=0,1,...,ne.$$
\end{defn}

Let $\mathbb{Q}_q$ be the unramified extension of $\mathbb{Q}_p$ of degree $a$,
and $\mathbb{Z}_q$ be its ring of integers.
Let $\zeta_p$ be a primitive $p$-th root of unity. Then $\mathbb{Z}_q[\zeta_p]$ and $\mathbb{Z}_p[\zeta_p]$
are rings of integers of $\mathbb{Q}_q(\zeta_p)$ and $\mathbb{Q}_p(\zeta_p)$, respectively.
Let $\pi$ be an element in an algebraic closure of $\mathbb{Q}_p$ such that $\pi^{p-1}=-p$.
By Krasner's lemma, we have $ \mathbb{Q}_p(\pi)=\mathbb{Q}_p(\zeta_p)$.
Adjoining the $e$-th root of $\pi$ in $\Omega$, say $\tilde{\pi}$,
we obtain totally ramified extensions of $\mathbb{Q}_q(\zeta_p)$ and
$\mathbb{Q}_p(\zeta_p)$, which are denoted by $K$ and $K_0$, respectively. Let $\mathbb{Z}_q[\tilde{\pi}]$ and
$\mathbb{Z}_p[\tilde{\pi}]$ denote the respective rings of integers of $K$ and $K_0$.
Let $\lambda$ be the Teichm$\ddot{\rm {u}}$ller representative of $\bar{\lambda}$ and $F(\lambda,x)$ be the Teichm${\rm \ddot{u}}$ller lifting of $\bar{F}(\bar{\lambda},x)$. Define
$$\mathcal{C}_0:=\{\sum_{u\in\mathbb{Z}^n}c_u\tilde{\pi}^{ew(u)}x^u: c_u\in \mathbb{Z}_q[\tilde{\pi}],
c_u\rightarrow 0 {\rm\ as\ } u\rightarrow \infty \}.$$
%Define set $\mathcal{O}_0$ by
%$$\mathcal{O}_0:=\Big\{\sum_{r=0}^{\infty}C_r\Lambda^{r/D}\pi^{W_{\Lambda}(r/D)}:
%C_r\in \mathbb{Z}_q[\tilde{\pi}],C_r\rightarrow 0\ {\rm as}\ r \rightarrow \infty  \Big\}$$
%with a valuation via
%$$\Big|\sum_{r=0}^{\infty}C_r\Lambda^{r/D}\pi^{W_{\Lambda}(r/D)}\Big|:=\sup_{r\ge 0}\{|C_r|\}.$$
% $\varphi$ ?be t
%The reduction map $\mod \tilde{\pi}$ maps $\mathcal{O}_0$ onto $S=\mathbb{F}_q[\Lambda^{1/D}]$ by sending
%$$ \sum_{r=0}^{\infty}C_r\Lambda^{r/D}\pi^{W_{\Lambda}(r/D)}\mapsto\sum_{r=0}^{\infty}\bar{C}_r\Lambda^{r/D}.$$
%Then the reduction map identifies the $\mathbb{F}_q$-algebras
%$$\mathcal{O}_0/ \tilde{\pi} \mathcal{O}_0\cong S.$$
%We now express the $p$-adic Banach space $\mathcal{C}_0$ concretely.
%Let $\gamma$ be a positive real number. Then we define
%$$\mathcal{C}_0(\gamma):=\Big\{\sum_{v\in M(\bar{f},\mu)}\zeta(v)\pi^{\gamma w(v)}\Lambda^{m(v)}x^v : \zeta(v)\in \mathcal{O}_0,\zeta(v)
%\rightarrow 0\ {\rm as}\ w(v) \rightarrow \infty \Big\}$$
%to be a $p$-adic Banach $\mathcal{O}_0$-algebra. Especially, we write $\mathcal{C}_0$ for $\mathcal{C}_0(1)$.
%Then the reduction map $\mod \tilde{\pi}$ taking
%$$\sum_{v\in M(\bar{f},\mu)}\zeta(v)\pi^{w(v)}\Lambda^{m(v)}x^v \mapsto \sum_{v\in M(\bar{f},\mu)}\bar{\zeta}(v)\Lambda^{m(v)}x^v,$$
%%induces an isomorphism of $S$-algebras
%$$\mathcal{C}_0/\tilde{\pi}\mathcal{C}_0\cong T.$$
Let $\theta(t):=\exp \pi(t-t^p)$. If we write
$\theta(t)=\sum_{i=0}^{\infty}\lambda_it^i$, it then follows from \cite{[Dwo62]} that
\begin{equation}\label{eq00}
{\rm ord}_p \lambda_i\ge \frac{p-1}{p^2}\cdot i
\end{equation}
for every $i\ge 0$.
Let $$F_0(x)=\theta(\lambda\prod_{i=1}^n x_i^{-d_i}) \prod_{i=1}^n \theta(x_i^{a_i})$$ and
$$F(x)=\prod_{j=0}^{a-1} F_0^{\sigma^j}(x^{p^j}),$$
where $\sigma\in Gal(K/K_0)$ is the Frobenius automorphism of $Gal(\mathbb{Q}_q/\mathbb{Q}_p)$
extended to $K$ by requiring $\sigma(\tilde{\pi})=\tilde{\pi}$ and $\sigma(\zeta_p)=\zeta_p$.

%Let $F(\Lambda,x)\in \mathbb{Z}_q[\Lambda,x_1^{\pm},x_2^{\pm}]$ be the Teichm${\rm \ddot{u}}$ller lifting of
%$\bar{F}(\Lambda,x)$. Then let
%$$F^{(0)}(\Lambda,x):=\pi F(\Lambda,x).$$
%$For $l=1,2$, let
%$$D_{l,\Lambda}:=x_l\frac{\partial}{\partial x_l}+x_l \frac{\partial F^{(0)}}{\partial x_l}.$$

We may define the Koszul complex $\Omega^{\bullet}(\mathcal{C}_0,\nabla(D))$ by letting
\begin{equation}\label{eq3}
\Omega^{i}(\mathcal{C}_0,\nabla(D)):=\bigoplus_{1\le j_1<...<j_i\le n} \mathcal{C}_0 \frac{d x_{j_1}}{x_{j_1}}\wedge \cdots\wedge\frac{d x_{j_i}}{x_{j_i}}
\end{equation}
with boundary map
$$\nabla(D)(\zeta \frac{d x_{j_1}}{x_{j_1}}\wedge\cdots\wedge \frac{d x_{j_i}}{x_{j_i}})=\Big(\sum_{l=1}^nD_{l}(\zeta)\frac{d x_{l}}{x_{l}}\Big)\wedge \frac{d x_{j_1}}{x_{j_1}}\wedge\cdots\wedge\frac{d x_{j_i}}{x_{j_i}},$$
where $$D_{l}=x_l\frac{\partial}{\partial x_l}+x_l\frac{\pi\partial F(\lambda,x)}{\partial x_l}.$$

%\begin{thm}(Theorem 3.1, \cite{HS17})\label{thm5}
%The complex $\Omega^{\bullet}(\mathcal{C}_0,\nabla(D^{(\Lambda)}))$ is acyclic except in top dimension $n$ and $H^n(\Omega^{\bullet}(\mathcal{C}_0,\nabla(D^{(\Lambda)})))$
%is a free $\mathcal{O}_0$-module of rank equal to $N$.
%Furthermore,
%$$\mathcal{C}_0=\sum_{(m(v);v)\in \bar{\bar{B}}}\mathcal{O}_0\pi^{w(v)}\Lambda^{m(v)}x^v \oplus \sum _{l=1}^nD_{l,\Lambda}(\mathcal{C}_0).$$
%\end{thm}

Set $$\alpha_0:=\sigma^{-1}\circ \psi\circ F_0$$
 and
$$\alpha:=\psi^a\circ F,$$
where $\psi$ is defined as
$$\psi(\sum A_vx^v)=\sum A_{pv}x^v.$$
Then by estimate (\ref{eq00}) and $p>2$ we conclude that $\alpha_0$ is an $\mathbb{Z}_q[\tilde{\pi}]$-semilinear morphism acting on $\mathcal{C}_0$,
and $\alpha$ maps $\mathcal{C}_0$
to $\mathcal{C}_{0}$ linearly over $\mathbb{Z}_q[\tilde{\pi}]$.
It follows from Serre \cite{Se} that the operators $\alpha^i$ and $\alpha_0^i$ acting on $\mathcal{C}_0$
have well-defined traces.

By (\cite{[Dwo62]}, equation(4.35))
the following communication law holds
\begin{equation}\label{eq21}
q D_{l}\circ \alpha=\alpha\circ D_{l}
\end{equation}
for $l=1,...,n$. Define an endomorphism $Frob^i$ on $\Omega^{i}(\mathcal{C}_0,\nabla(D))$ by
\begin{equation}\label{eq14}
Frob^i:= \bigoplus _{1\le j_1<...<j_i\le n} q^{n-i}\alpha \frac{d x_{j_1}}{x_{j_1}}\wedge\cdots\wedge\frac{d x_{j_i}}{x_{j_i}}.
\end{equation}
The commutation rule (\ref{eq21}) ensures that (\ref{eq14}) defines a chain map on the Koszul complex $\Omega^{\bullet}(\mathcal{C}_{0},\nabla (D))$.

 Then by the Dwork's trace formula, we have
\begin{align*}
S_k(\bar{F},\bar{\lambda})=\sum_{i=0}^n(-1)^iTr(H^i(Frob)^k\ |\ H^i(\mathcal{C}_{0}, \nabla (D))).
\end{align*}

Using the same argument to Theorem \ref{thm0} as \cite{HS17}, \cite{HS14}, or going back to \cite{AS89} or \cite{[PM70]}, we have the following result.
\begin{prop}\label{prop1}
 The cohomology of $\Omega^{\bullet}(\mathcal{C}_{0}, \nabla(D))$ is acyclic except in top dimension $n$, and $H^n(\Omega^{\bullet}(\mathcal{C}_{0}, \nabla(D))) $ is
a free $\mathbb{Z}_q[\tilde{\pi}]$-module of rank equal to $n!Vol(\Delta(\bar{F}))$. Furthermore,
 $$\mathcal{C}_{0}=\sum_{v\in B}\mathbb{Z}_q[\tilde{\pi}]\tilde{\pi}^{ew(v)}x^v\oplus \sum_{l=1}^n D_{l}\mathcal{C}_{0},$$
 where $B$ is defined as Theorem \ref{thm0}.
\end{prop}
It follows that
$$S_k(\bar{F},\bar{\lambda})=(-1)^n Tr\big(H^n (Frob)^k| H^n(\mathcal{C}_{0}, \nabla (D))\big).$$
Hence
$$L(\bar{F}(\bar{\lambda},x), T)^{(-1)^{n-1}}=\det (1-FrobT|H^n(\mathcal{C}_{0}, \nabla (D))).$$

It follows from \cite{AS89} that the Newton polygon of $L(\bar{F}(\bar{\lambda},x), T)^{(-1)^{n-1}}$
lies over the Newton polygon (using $ord_{q}$) of
\begin{equation}\label{eq27}
\prod_{v\in B}(1-q^{w(v)}T).
\end{equation}
 In other words,
$$NP(\bar{F})\ge HP(\Delta(\bar{F})).$$
In particular, the polynomial $\bar{F}$ is called ordinary if $NP(\bar{F})= HP(\Delta(\bar{F}))$.

\section{Lower bound for Newton polygon}

\subsection{Basis}

For $n=1$, we let $\Delta_0=\{a_1\}$ and
$\Delta_1=\{-d_1\}$.
For $n\ge 2$, we define the following notations.
For each $j=1,...,n$, by $A_j\in \mathbb{Z}^n$ we denote the vector such that the $j$-th component is $a_j$ and the other components are 0.
Let $\gamma=(-d_1,\cdots,-d_n)$. Then
$$supp(\bar{F})=\{\gamma\}\cup\{A_j\}_{j=1}^n.$$
Let $\Delta_0$ denote the convex closure generated by the lattice points $A_j(1\le j\le n)$.
For $i=1,...,n$, let $\Delta_i$ be the convex closure generated by $A_j(1\le j\neq i\le n)$ and $\gamma$. We define
$$C(\Delta_0):=\{u\in \mathbb{Z}^n: u_i\ge 0\ {\rm for}\ i=1,...,n\}.$$
For $i=1,\cdots,n$, we let
 $$C(\Delta_i):=\{u\in \mathbb{Z}^n: u_i\le 0\ {\rm and}\ u_j-\frac{d_j}{d_i}u_i\ge 0 {\rm\ for\ any}\ 1\le j\neq i\le n\}.$$
 From the definition, we can see that if there exists an integer $i$ such that $u_i=0$ or there is at least one
pair of integers $1\le i\neq j\le n$
such that $u_{j}=\frac{d_{j}}{d_{i}}u_{i}$ and $u_i<0$, then $u$ is cofacial to all elements in $Supp(\bar{F})$ with respect to $\Delta(\bar{F})$.

Recall that for $n=1$, $\mathcal{B}=\{v\in \mathbb{Z}: -d_1<v\le a_1\} $.
For $n\ge 2$, the set $\mathcal{B}=\{(v_1,\cdots, v_n)\in \mathbb{Z}^n\}$ such that
$-d_i<v_i\le a_i$ for all $i=1,\cdots,n$, and $(v_i,v_j)$ satisfies that
\begin{align*}
 \frac{d_j}{d_i}(v_i-a_i)\le v_j< \frac{d_j}{d_i}v_i+a_j
\end{align*}
for $1\le i<j\le n$. And $\bar{B}=\{x^v:v\in \mathcal{B}\}$.

\begin{lem}\label{lemn2} Let $v:=(v_1,...,v_n)\in \mathbb{Z}^n$ such that $-d_i<v_i\le a_i$ for $i=1,...,n$.
Then there exists $u=(u_1,...,u_n)\in \mathcal{B}$ and $c\in \mathbb{F}_q$ such that
$$x^v\equiv cx^u \mod \sum_{l=1}^n\bar{F}_l\bar{R}.$$
\end{lem}
\begin{proof}
Evidently, Lemma \ref{lemn2} is true for $n=1$. In what follows, we let $n\ge 2$.
If $v\in \mathcal{B}$, then Lemma \ref{lemn2} is true. Hence we let $v\notin \mathcal{B}$.
Then there is at least one pair of integers $(i,j)$
such that (\ref{eq1}) does not hold.
Define the set $M_v$ as follows:
$$M_v:=\Big\{(i,j): 1\le i<j\le n\ {\rm\ such\ that}\ v_j\ge\frac{d_j}{d_i}v_i+a_j\ {\rm or}\ v_j<\frac{d_j}{d_i}(v_i-a_i) \Big\}.$$

We claim that for each positive integer $k$, if $|M_v|=k$, then there are $c\in \mathbb{F}_q$ and $v'\in \mathbb{Z}^n$ such that
$$x^v\equiv cx^{v'}\mod \sum_{l=1}^n\bar{F}_l\bar{R},$$
$-d_i<v'_{i}\le a_i$ for $i=1,...,n$ and $|M_{v'}|\le k-1$.
In particular, if $|M_v|=1$, then $v'\in \mathcal{B}$.

Suppose $M_v=\{(i_1,j_1),...,(i_k,j_k)\}$. Let
$$M_1=\{(i,j)\in M_v: v_j\ge \frac{d_{j}}{d_{i}}v_{i}+a_{j}\}$$
and
$$M_2=\{(i,j)\in M_v: v_j< \frac{d_{j}}{d_{i}}(v_{i}-a_{i})\}.$$
%Let $J_1=\{j: (i,j)\in M_1\}$ and $J_2=\{j: (j,i)\in M_2\}$.
Let $$J=\{j: (i,j)\in M_1\ {\rm or}\ (j,i)\in M_2\}.$$
%Let $J=J_1\cup J_2$.
Let $j_0$ be the largest integer in $J$ such that
$$\frac{v_{j_0}-a_{j_0}}{d_{j_0}}=\max\{ \frac{v_{j}-a_{j}}{d_{j}}: j\in J\} .$$
%Let $I_1=\{i: (i,j_0)\in M_1\}$ and $I_2=\{i:(j_0,i)\in M_2\}$. Let $I=I_1\cup I_2$
Let $$I=\{i: (i,j_0)\in M_1\ {\rm or}\ (j_0,i)\in M_2\}$$
and $i_0$ be the least integer of $I$ such that
$$\frac{v_{i_0}}{d_{i_0}}=\min \big\{\frac{v_{i}}{d_{i}}: i\in I \big\} .$$
Clearly, one has $(i_0,j_0)\in M_1$ or $(j_0,i_0)\in M_2$. Consider the following two cases.

{\sc Case 1.} $(i_0,j_0)\in M_1$. That is, $v_{j_0}\ge \frac{d_{j_0}}{d_{i_0}}v_{i_0}+a_{j_0}$.
Then $v_{i_0}\le 0$. If $v-A_{j_0}$ is not cofacial to $A_{i_0}$,
then $$\bar{F}_{j_0}\cdot x^{v-A_{j_0}}=a_{j_0}x^v-d_{j_0}\bar{\lambda}x^{v-A_{j_0}-\gamma}$$
and
$$\bar{F}_{i_0}\cdot x^{v-A_{j_0}}=0-d_{i_0}\bar{\lambda}x^{v-A_{j_0}-\gamma}.$$
Hence
$$d_{i_0}a_{j_0}x^v\equiv0\mod \sum_{l=1}^n\bar{F}_l\bar{R}.$$

If $v-A_{j_0}$ is cofacial to $A_{i_0}$, then
$$\bar{F}_{j_0}\cdot x^{v-A_{j_0}}=a_{j_0}x^v-d_{j_0}\bar{\lambda}x^{v-A_{j_0}-\gamma}$$
and
$$\bar{F}_{i_0}\cdot x^{v-A_{j_0}}=a_{i_0}x^{v-A_{j_0}+A_{i_0}}-d_{i_0}\bar{\lambda}x^{v-A_{j_0}-\gamma}.$$
Hence
$$d_{i_0}a_{j_0}x^v\equiv d_{j_0}a_{i_0}x^{v-A_{j_0}+A_{i_0}}\mod \sum_{l=1}^n\bar{F}_l\bar{R}.$$
Let $v^{(1)}=v-A_{j_0}+A_{i_0}$.
One has that $$v^{(1)}_{j_0}=v_{j_0}-a_{j_0}\ge \frac{d_{j_0}}{d_{i_0}}(v_{i_0}+a_{i_0}-a_{i_0})=\frac{d_{j_0}}{d_{i_0}}(v^{(1)}_{i_0}-a_{i_0}).$$
One can check that $-d_{j_0}<v^{(1)}_{j_0}=v_{j_0}-a_{j_0}\le a_{j_0}$ and $-d_{i_0}<v^{(1)}_{i_0}=v_{i_0}+a_{i_0}\le a_{i_0}$.
Now we prove that if $(i,j_0)\notin M_v$, then $(i,j_0)\notin M_{v^{(1)}}$. Since $(i,j_0)\notin M_v$,
one has that
$$\frac{d_{j_0}}{d_i}(v_i-a_i)\le v_{j_0}<\frac{d_{j_0}}{d_i}v_i+a_{j_0}.$$
Clearly, one has $v_{j_0}-a_{j_0}< \frac{d_{j_0}}{d_i}v_i+a_{j_0}$. If $v_{j_0}-a_{j_0}< \frac{d_{j_0}}{d_i}(v_i-a_i)$, then
$v_{i_0}<\frac{d_{i_0}}{d_i}(v_i-a_i)$.
If $i>i_0$, then $(i_0,i)\in M_1 $.
If $i<i_0$, then $(i,i_0)\in M_2$. It follows from the definition of $j_0$ that $\frac{v_{j_0}-a_{j_0}}{d_{j_0}}\ge\frac{v_i-a_i}{d_i}$.
But $\frac{v_{j_0}-a_{j_0}}{d_{j_0}}<\frac{v_i-a_i}{d_i}$, which is a contradiction.
Hence
$$v_{j_0}-a_{j_0}\ge \frac{d_{j_0}}{d_i}(v_i-a_i).$$
 That is, $(i,j_0)\notin M_{v^{(1)}}$.
On the other hand, if $(j_0,i)\notin M_v$, then
$$\frac{d_i}{d_{j_0}}(v_{j_0}-a_{j_0})\le v_i<\frac{d_i}{d_{j_0}}v_{j_0}+a_i.$$
Clearly, $\frac{d_i}{d_{j_0}}(v_{j_0}-2a_{j_0})\le v_i$. If $v_i\ge \frac{d_i}{d_{j_0}}(v_{j_0}-a_{j_0})+a_i$,
then $v_i-a_i\ge \frac{d_i}{d_{i_0}}v_{i_0}$.
Note that $i>j_0>i_0$, then $(i_0,i)\in M_1 $. By the definition, we have $\frac{v_{j_0}-a_{j_0}}{d_{j_0}}<\frac{v_i-a_i}{d_i}$, which is a contradiction.
Hence
$$v^{(1)}_{j_0}=v_{j_0}-a_{j_0}\ge \frac{d_{j_0}}{d_i}(v_i-a_i).$$
That is, $(j_0,i)\notin M_{v^{(1)}}$.

We now show that if $(i,i_0)\notin M_v$, then $(i,i_0)\notin M_{v^{(1)}}$.
Since $(i,i_0)\notin M_v$, one has that
$$\frac{d_{i_0}}{d_i}(v_i-a_i)\le v_{i_0}<\frac{d_{i_0}}{d_i}v_i+a_{i_0}.$$
Clearly, $\frac{d_{i_0}}{d_i}(v_i-a_i)\le v_{i_0}+a_{i_0}$.
If $v_{i_0}+a_{i_0}\ge \frac{d_{i_0}}{d_i}v_i+a_{i_0}$, then
$$v_{j_0}-a_{j_0}\ge \frac{d_{j_0}}{d_{i_0}}v_{i_0}\ge \frac{d_{j_0}}{d_{i}}v_{i}. $$
Note that $i<i_0<j_0$. It follows that $(i,j_0)\in M_1$.
Hence $\frac{v_i}{d_i}>\frac{v_{i_0}}{d_{i_0}}$ by the definition of $i_0$, which contradicts to
$v_{i_0}\ge \frac{d_{i_0}}{d_i}v_i$.
Thus $v_{i_0}+a_{i_0}< \frac{d_{i_0}}{d_i}v_i+a_{i_0}$. It follows that $(i,i_0)\notin M_{v^{(1)}}$.
By the similar argument, we can prove that if $(i_0,i)\notin M_v$, then $(i_0,i)\notin M_{v^{(1)}}$.
Hence $|M_{v^{(1)}}|\le |M_v|$.

{\sc Case 2.} $(j_0,i_0)\in M_2$. That is, $v_{i_0}<\frac{d_{i_0}}{d_{j_0}}(v_{j_0}-a_{j_0})$.
If $v-A_{j_0}$ is not cofacial to $A_{i_0}$,
then $$\bar{F}_{j_0}\cdot x^{v-A_{j_0}}=a_{j_0}x^v-d_{j_0}\bar{\lambda}x^{v-A_{j_0}-\gamma}$$
and
$$\bar{F}_{i_0}\cdot x^{v-A_{j_0}}=0-d_{i_0}\bar{\lambda}x^{v-A_{j_0}-\gamma}.$$
Hence $$d_{i_0}a_{j_0}x^v\equiv0\mod \sum_{l=1}^n\bar{F}_l\bar{R}.$$

If $v-A_{j_0}$ is cofacial to $A_{i_0}$, then
$$\bar{F}_{j_0}\cdot x^{v-A_{j_0}}=a_{j_0}x^v-d_{j_0}\bar{\lambda}x^{v-A_{j_0}-\gamma}$$
and
$$\bar{F}_{i_0}\cdot x^{v-A_{j_0}}=a_{i_0}x^{v-A_{j_0}+A_{i_0}}-d_{i_0}\bar{\lambda}x^{v-A_{j_0}-\gamma}.$$
Hence $$d_{i_0}a_{j_0}x^v\equiv d_{j_0}a_{i_0}x^{v-A_{j_0}+A_{i_0}}\mod \sum_{l=1}^n\bar{F}_l\bar{R}.$$
Let $v^{(1)}=v-A_{j_0}+A_{i_0}$. One has that
$$v^{(1)}_{i_0}-a_{i_0}=v_{i_0}+a_{i_0}-a_{i_0}<\frac{d_{i_0}}{d_{j_0}}(v_{j_0}-a_{j_0})=\frac{d_{i_0}}{d_{j_0}}v^{(1)}_{j_0}.$$
By the similar argument as Case 1, we can prove that $|M_{v^{(1)}}|\le|M_v|$.

For $(i_0,j_0)\in M_1$, if $v_{j_0}-a_{j_0}< \frac{d_{j_0}}{d_{i_0}}(v_{i_0}+a_{i_0})+a_{{j_0}}$, then we let $v'=v^{(1)}$ and
$(i_0,j_0)\notin M_{v'}$.
For $(j_0,i_0)\in M_2$, if $v_{j_0}-a_{j_0}\le \frac{d_{j_0}}{d_{i_0}}(v_{i_0}+a_{i_0})+a_{{j_0}}$, then we let $v'=v^{(1)}$ and $(j_0,i_0)\notin M_{v'}$.
Hence $|M_{v'}|< |M_v|$.

If $v_{j_0}-a_{j_0}\ge \frac{d_{j_0}}{d_{i_0}}(v_{i_0}+a_{i_0})+a_{{j_0}}$ for $(i_0,j_0)\in M_1$, or $\frac{d_{i_0}}{d_{j_0}}(v_{j_0}-2a_{j_0})> v_{i_0}+a_{i_0}$ for $(j_0,i_0)\in M_2$, then we repeat the above process for $v^{(1)}=v-A_{j_0}+A_{i_0}$.
Suppose $$M_{v^{(1)}}=\{(i^{(1)}_1,j^{(1)}_1),...,(i^{(1)}_{k^{(1)}},j^{(1)}_{k^{(1)}})\}$$ with
$k^{(1)}\le k$. Let
$$M^{(1)}_1=\{(i,j)\in M_{v^{(1)}}: v^{(1)}_j\ge \frac{d_{j}}{d_{i}}v^{(1)}_{i}+a_{j}\}$$
and
$$M^{(1)}_2=\{(i,j)\in M_{v^{(1)}}: v^{(1)}_j< \frac{d_{j}}{d_{i}}(v^{(1)}_{i}-a_{i})\}.$$
%Let $J^{(1)}_1=\{j: (i,j)\in M^{(1)}_1\}$ and $J^{(1)}_2=\{j: (j,i)\in M^{(1)}_2\}$. Let $J^{(1)}=J^{(1)}_1\cup J^{(1)}_2$.
Let
$$J^{(1)}=\{j: (i,j)\in M^{(1)}_1\ {\rm or}\ (j,i)\in M^{(1)}_2\}.$$
Let $j^{(1)}_0$ be the largest element in $J^{(1)}$ such that
$$\frac{v^{(1)}_{j^{(1)}_0}-a_{j^{(1)}_0}}{d_{j^{(1)}_0}}=\max\{ \frac{v^{(1)}_{j}-a_{j}}{d_{j}}: j\in J^{(1)}\} .$$
%Let $I^{(1)}_1=\{i: (i,j_0)\in M^{(1)}_1\}$ and $I^{(1)}_2=\{i:(j_0,i)\in M^{(1)}_2\}$. Let $I^{(1)}=I^{(1)}_1\cup I^{(1)}_2$
Let $$I^{(1)}=\{i: (i,j^{(1)}_0)\in M^{(1)}_1 \ {\rm or}\ (j^{(1)}_0,i)\in M^{(1)}_2\}$$
and
$i^{(1)}_0$ be the least element of $I^{(1)}$ such that
$$\frac{v^{(1)}_{i^{(1)}_0}}{d_{i^{(1)}_0}}=\min \Big\{\frac{v^{(1)}_{i}}{d_{i}}: i\in I^{(1)} \Big\} .$$
Clearly, one has $(i^{(1)}_0,j^{(1)}_0)\in M^{(1)}_1$ or $(j^{(1)}_0,i^{(1)}_0)\in M^{(1)}_2$.
Note that $-d_i<v_i\le a_i$ for $i=1,...,n$. We observe that
$$\frac{v^{(1)}_{j^{(1)}_0}-a_{j^{(1)}_0}}{d_{j^{(1)}_0}}\le\frac{v_{j_0}-a_{j_0}}{d_{j_0}}$$
and
$$\frac{v^{(1)}_{i^{(1)}_0}}{d_{i^{(1)}_0}}\ge \frac{v_{i_0}}{d_{i_0}}.$$
Hence after finite steps one can obtain $v'\in \mathbb{Z}^n$ such that
$x^v\equiv cx^{v'}\mod \sum_{l=1}^n\bar{F}_l\bar{R}$,
$-d_i<v'_{i}\le a_i$ for $i=1,...,n$ and $|M_{v'}|< |M_v|$.

Particularly, if $|M_v|=1$, then there are $c\in \mathbb{F}_q$ and $v'\in \mathbb{Z}^n$ such that $x^v\equiv cx^{v'}\mod \sum_{l=1}^n\bar{F}_l\bar{R}$,
$-d_i<v'_{i}\le a_i$ for $i=1,...,n$ and $|M_{v'}|=0$. That is, $v'\in \mathcal{B}$.
This finishes the proof of the claim.

Since $v\notin \mathcal{B}$, there is at least one pair of integers $i,j$ with $1\le i<j\le n$ satisfying
$v_j<\frac{d_j}{d_i}(v_i-a_i)$
or $v_j\ge\frac{d_j}{d_i}v_i+a_j$. %Since $v_j\ge \frac{d_j}{d_i}v_i$, one has $v_j\ge\frac{d_j}{d_i}v_i+a_j$.
We prove Lemma \ref{lemn2} by induction on $|M_v|$.
For $|M_v|=1$, it follows from the claim that
 Lemma \ref{lemn2} is true when $|M_v|=1$.

Let $m$ be a positive integer with $m\ge 2$. Suppose that Lemma \ref{lemn2} holds for $|M_v|\le m-1$.
We consider $|M_v|=m$.
It follows from the claim that there are
$c'\in \mathbb{F}_q$ and $u'\in \mathbb{Z}^n$ such that
$$x^v\equiv c'x^{u'}\mod \sum_{l=1}^n\bar{F}_l\bar{R},$$
$-d_i<u'_{i}\le a_i$ for $i=1,...,n$ and $|M_{u'}|\le m-1$.
By the hypothesis, there are $u=(u_1,...,u_n)\in \mathcal{B}$ and $c\in \mathbb{F}_q$ such that
$$x^v\equiv c'x^{u'}\equiv cc'x^u \mod \sum_{l=1}^n\bar{F}_l\bar{R}.$$ Hence Lemma \ref{lemn2} holds for $|M_v|=m$.
This finishes the proof of Lemma \ref{lemn2}.

%{\sc Case 1. $v\in C(\Delta_0)$}. Then there is at least one pair of integers $1\le i<j\le n$ such that $v_j<\frac{d_j}{d_i}(v_i-a_i)$
%or $v_j\ge\frac{d_j}{d_i}v_i+a_j$. Clearly,  it is impossible that $v_j<\frac{d_j}{d_i}(v_i-a_i)$ since $v_j\ge 0$.
%We proceed by induction on size of $M_v$.
%Suppose $|M_v|=1$. It follows from (2) of claim that Lemma \ref{lemn2} holds in this case.
%Let $k\ge$1. Suppose Lemma \ref{lemn2} holds for the $|M_v|\le k-1$ case. Let $|M_v|=k $.
%It follows from (1) of the claim that there is $v'$ such that $|M_{v'}|\le k-1$.
%By the hypothesis, there are $u=(u_1,...,u_n)\in B$ and $c\in \mathbb{F}_q$ such that
%$$x^v\equiv cx^u \mod \sum_{l=1}^n\bar{F}_l\bar{R}.$$ Hence Lemma \ref{lemn2} holds for $|M_v|=k$ case.
%Thus Lemma \ref{lemn2} is true in this case.

%{\sc Case 2.} $v\in C(\Delta_l)$ for $l\in \{1,...,n\}$.
\end{proof}

\begin{lem}\label{thm3}
Let $v\in \mathbb{Z}^n$. Then $x^v$ is a linear combination of elements in $\bar{B}$ over $\mathbb{F}_q$ modulo $\sum_{l=1}^n\bar{F}_{l}\bar{R}$.
\end{lem}
\begin{proof}
 First, we prove Lemma \ref{thm3} is true for $n=1$. If $v>a_1$, then $v-a_1$ is not cofacial to $-d_1$.
Hence $\bar{F}_1(x^{v-a_1})=a_1x^v$ and
$$a_1x^v\equiv0 \mod \bar{F}_{1}\bar{R}.$$
If $v< -d_1$, then $v+d_1$ is not cofacial to $a_1$.
Thus $\bar{F}_1(x^{v+d_1})=-d_1x^v$. That is,
$$d_1x^v\equiv0 \mod \bar{F}_{1}\bar{R}.$$
If $v=-d_1$, then $v+d_1$ is cofacial to $a_1$ and $d_1$.
It follows that $\bar{F}_1(x^{v+d_1})=a_1x^{a_1}-d_1x^v$ and
$$a_1x^{a_1}\equiv d_1x^v \mod \bar{F}_{1}\bar{R}.$$
Hence Lemma \ref{thm3} is true when $n=1$.

In the following, we let $n\ge 2$ and divide the proof into two cases.

{\sc Case 1.} $v\in C(\Delta_0)$.
Let
$$S_v:=\{k: v_k=0\} .$$
 If $v_i\le a_i$ for all $i$, it then follows from Lemma \ref{lemn2} that Lemma \ref{thm3} is true.
Hence in what follows, we assume that there is at least one integer $j$ such that $v_j>a_j$.
If $|S_v|=0$, then
$v_i>0$ for all $i$. Let $j$ be the integer such that $v_j>a_j$.
One has that $v-A_j$ is not cofacial to $\gamma$.
Then $F_j\cdot x^{v-A_j}=x^v $.
Hence $x^v\equiv 0 \mod \sum_{l=1}^n\bar{F}_{l}\bar{R}.$
Thus Lemma \ref{thm3} is true if $|S_v|=0$.

If $|S_v|=1$, then there is one integer $k$ such that $v_k=0$.
Let $j$ be the integer such that $v_j>a_j$. Then $v-A_j$ is cofacial to all elements in $supp(\bar{F})$.
Hence
$$\bar{F}_j\cdot x^{v-A_j}=a_jx^v-d_j\bar{\lambda}x^{v-A_j-\gamma}$$
and
$$\bar{F}_k\cdot x^{v-A_j}=a_kx^{v-A_j+A_k}-d_k\bar{\lambda}x^{v-A_j-\gamma}.$$
Hence $$a_jd_kx^v\equiv a_kd_jx^{v-A_j+A_k} \mod \sum_{l=1}^n\bar{F}_{l}\bar{R}.$$
We can see that $v'=v-A_j+A_k\in C(\Delta_0)$ and $|S_{v'}|=0$. Hence we conclude that Lemma \ref{thm3} is true for $|S_v|=1$.
  Then by induction on $|S_v|$, we can prove Lemma \ref{thm3} is true in this case.

{\sc Case 2.} $v\in C(\Delta_i)$ for $i\in \{1,...,n\}$. Then $v_i\le 0$ and $v_j\ge \frac{d_j}{d_i} v_i$ for $1\le j\neq i\le n$.
If $v_i=0$, then $v_j\ge 0$ for $j=1,...,n$. Hence $v\in C(\Delta_0)$. Lemma \ref{thm3} has been proved to be true when $v_i=0$.
In what follows, we let $v_i<0$.

Assume that $v_j> \frac{d_j}{d_i} v_i$ for $1\le j\neq i\le n$. If there is an integer $k$ such that $k\neq i$ and $v_k>a_k$, then $v-A_k\in C(\Delta_i)$ and $v-A_k$ is not cofacial to $A_i$.
Then
$$\bar{F}_k\cdot x^{v-A_k}=a_kx^v-d_k\bar{\lambda}x^{v-A_k-\gamma}$$
and
$$\bar{F}_i\cdot x^{v-A_k}=-d_i\bar{\lambda}(v-A_k-\gamma).$$
Hence
$$x^v\equiv 0 \mod \sum_{l=1}^n\bar{F}_{l}\bar{R}.$$
In what follows, we suppose $v_j\le a_j$ for all $j=1,\cdots,n$.
If $-d_i<v_i<0$, then $v_j>-d_j$ for all $j\neq i$. It follows from Lemma \ref{lemn2} that
Lemma \ref{thm3} is true for $-d_i<v_i<0$.
If $v_i< -d_i$, then $v_i+d_i< 0$.
It is easy to check that $v_j+d_j> \frac{d_j}{d_i} (v_i+d_i)$.
Hence $(v_1+d_1,\cdots,v_n+d_n)\in C(\Delta_i)$ and is not cofacial to $A_i$.
Then $$F_i\cdot x^{v+\gamma}=-d_i\bar{\lambda}x^v .$$
Hence $x^v\equiv 0 \mod \sum_{l=1}^n\bar{F}_{l}\bar{R}.$
Thus Lemma \ref{thm3} is true for $v_i< -d_i$.
If $v_i=-d_i$, then $v_i+d_i=0$ and $$v_j+d_j> \frac{d_j}{d_i} (v_i+d_i)= 0$$ for all $j=1,\cdots,n$.
 Hence $v+\gamma$ is cofacial to all elements in $supp(\bar{F})$.
 Then $$F_i\cdot x^{v+\gamma}=a_ix^{v+\gamma+A_i}-d_i\bar{\lambda}x^v .$$
Evidently, $v+\gamma+A_i\in C(\Delta_0)$ and $$a_ix^{v+\gamma+A_i}\equiv d_i\bar{\lambda}x^v \mod \sum_{l=1}^n\bar{F}_l\bar{R}.$$
Hence by Case 1 one has that Lemma \ref{thm3} is true for $v_i=-d_i$.

Assume that $v_j= \frac{d_j}{d_i} v_i$ for some $1\le j\neq i\le n$.
If $-d_i\le v_i<0$, then $v_k\ge-d_k$ for all $1\le k\neq i\le n$.
Then
$$F_{i}\cdot x^{v+\gamma}=a_{i}x^{v+\gamma+A_{i}}-d_{i}\bar{\lambda}x^v .$$
Clearly, $v+\gamma+A_{i}\in C(\Delta_0)$. It follows from Case 1 that Lemma \ref{thm3} is true.

Suppose $v_i<-d_i$. Let $J=\{j_1,\cdots,j_m\}$ such that $v_{j_l}= \frac{d_{j_l}}{d_i} v_i$ for $l=1,\cdots,m$. Then $v_{j_l}< -d_{j_l}$ for
$l=1,\cdots,m$.
Clearly, $v_{j_l}+d_{j_l}= \frac{d_{j_l}}{d_i} (v_i+d_i)$ for $l=1,\cdots,m$.
Then using $F_{j_1}$ to act on $x^{v+\gamma}$, we have
$$F_{j_1}\cdot x^{v+\gamma}=a_{j_1}x^{v+\gamma+A_{j_1}}-d_{j_1}\bar{\lambda}x^v .$$
 Clearly, $v_{j_1}+d_{j_1}+a_{j_1}>\frac{d_{j_1}}{d_i}(v_i+d_i)$. If $v_i+d_i\ge-d_i$, then we have done.
 If $v_i+d_i<-d_i$, then we continue use $F_{j_2}$ to act on $x^{v+2\gamma+A_{j_1}}$.
 The process will stop after finite steps since $J$ is a finite set.
After at most $m$ steps, we get $$v_{j_l}+md_{j_l}+a_{j_l}>\frac{d_{j_l}}{d_i}(v_i+md_i)$$ for $l=1,\cdots,m$.
Then by the above discussion that Lemma \ref{thm3} holds for $x^{v+m\gamma+\sum_{l=1}^m A_{j_l}}$. So does $x^v$.
Hence Lemma \ref{thm3} is true in this case.

This finishes the proof of Lemma \ref{thm3}.
\end{proof}

To prove Theorem \ref{thm10} is true, it remains to show that the
number of lattices in $\mathcal{B}$ is equal to $n!Vol \Delta(\bar{F})$.
For $n\ge 2$, we denote $v=(v_1,\cdots,v_n)$, and
define sets
$$\mathcal{A}:=\{v\in \mathbb{Z}^n: (v_i,v_j)\ {\rm satisfies}\ (\ref{eq1})\ {\rm for}\ 1\le i<j \le n-1, -d_i<v_i\le a_i\ {\rm for}\ 1\le i\le n\}$$
and
$$\mathcal{A}_0:=\{v\in \mathcal{A}: v\notin \mathcal{B}\}.$$
That is, if $v\in \mathcal{A}_0$, then there is an integer $1\le i\le n-1$ such that $v_n\ge \frac{d_n}{d_i}v_i+a_n$ or
$v_n< \frac{d_n}{d_i}(v_i-a_i)$.
Let $$\bar{\mathcal{T}}:=\{v\in \mathcal{A}:\ \exists\ 1\le i\le n-1\ \ {\rm such\ that}\ v_n\ge\frac{d_n}{d_i}v_i+a_n \} $$
and
\begin{align*}
\mathcal{S}:=\{v\in \mathcal{A}: \ \exists\ 1\le i\le n-1\ {\rm such\ that}\    \frac{d_n}{d_i}(v_i-a_i)>v_n
%(v_j,v_n) \ {\rm satisfies}\  (\ref{eq1}) \ {\rm for\ all}\ j\neq i
\}.
\end{align*}
Hence $\mathcal{A}_0=\bar{\mathcal{T}}\cup \mathcal{S}$. Suppose $v\in \mathcal{S}$ and $i$ is the integer such that $v_n<\frac{d_n}{d_i}(v_i-a_i)$.
For any integer $j$ such that $1\le i<j\le n-1$, one has $\frac{d_j}{d_i}(v_i-a_i)\le v_j<\frac{d_j}{d_i}v_i+a_j$,
then by $v_n<\frac{d_n}{d_i}(v_i-a_i)$, we conclude that $v_n<\frac{d_n}{d_j}v_j+a_n$.
Similarly, we can show for $j<i\le n-1$, one has $v_n<\frac{d_n}{d_j}v_j+a_n$. Hence $v\notin \bar{\mathcal{T}}$ and
$\bar{\mathcal{T}}\cap \mathcal{S}=\emptyset$.
Let
$$\mathcal{M}=\{v\in \mathcal{A}: a_n-d_n<v_n\le a_n,\ \exists\ 1\le i\le n-1\ {\rm such\ that}\ -d_i<v_i\le 0\}$$
and
\begin{align*}
\mathcal{T}:=\{v\in \mathcal{A}: a_n-d_n<v_n\le a_n,\ &\exists\ 1\le i\le n-1\ {\rm such\ that}\ -d_i<v_i\le 0\\
& {\rm and}\
v_n< \frac{d_n}{d_j}v_j+a_n \ {\rm for\ all}\ j=1,\cdots,n-1
\}.
\end{align*}
Hence $\mathcal{T}\subset \mathcal{M}$.
If $v\in \bar{\mathcal{T}}$, then there is an integer $i$ such that $v_n\ge\frac{d_n}{d_i}v_i+a_n$.
 One can check that $-d_i<v_i\le 0$ and $a_n-d_n<v_n\le a_n$.
Then $\bar{\mathcal{T}}\subset \mathcal{M}$ and $\mathcal{M}=\mathcal{T}\cup \bar{\mathcal{T}}$.

The number of lattice points in $\mathcal{B}$ can be obtained by counting points in $\mathcal{A}$ and $\mathcal{A}_0$. Specifically, $|\mathcal{B}|=|\mathcal{A}|-|\mathcal{A}_0|$.
In the following, we first show  that $|\mathcal{A}_0|=|\mathcal{M}|$
by building a bijection between $\mathcal{T}$ and $\mathcal{S}$, then get $|\mathcal{M}|$.

For $v\in \mathcal{T}$, let $\mathcal{I}_v:=\{i: -d_i<v_i\le 0\}$. Let $i_0$ be the least element of $\mathcal{I}_v$ such that
$$\frac{v_{i_0}}{d_{i_0}}=\min\big\{\frac{v_i}{d_i}: i\in \mathcal{I}_v\big\}.$$
Define $\psi: \mathcal{T} \rightarrow \mathcal{S} $ by
$$\psi(v):=v+A_{i_0}-A_n.$$
For $u\in \mathcal{S}$, let $\mathcal{J}_u:=\{j: \frac{d_n}{d_{j}}(u_{j}-a_{j})>u_n\}$.
Let $j_0$ be the largest element of $\mathcal{J}_u$ such that
$$\frac{u_{j_0}-a_{j_0}}{d_{j_0}}=\max \big\{ \frac{u_j-a_j}{d_j}: j\in \mathcal{J}_u\big\} .$$
Define $\varphi: \mathcal{S} \rightarrow \mathcal{T} $ by
$$\varphi(u)=u-A_{j_0}+A_n. $$

Then we have
\begin{lem}\label{lem3}
$\psi(\mathcal{T})\subseteq \mathcal{S}$ and $\varphi(\mathcal{S})\subseteq \mathcal{T}$.
\end{lem}

\begin{proof}
Let $v\in \mathcal{T}$.
Let $u=(u_1,...,u_n)=\psi(v)$. Clearly,
$$u_n=v_n-a_n<\frac{d_n}{d_{i_0}}(v_{i_0}+a_{i_0}-a_{i_0})=\frac{d_n}{d_{i_0}}(u_{i_0}-a_{i_0}).$$
So to show $\psi(\mathcal{T})\subseteq \mathcal{S}$, it remains to show that $u\in \mathcal{A}$. Obviously, $-d_n<u_n\le a_n$
and $-d_{i_0}<u_{i_0}\le a_{i_0}$.
For $i<n$, if $i_0\neq i$, then $-d_i<u_i=v_i\le a_i$.
We also have (\ref{eq1}) is true for $i_0\neq i,j<n$.

For $1\le i<i_0$, since $v\in \mathcal{T}$, one has
$$\frac{d_{i_0}}{d_{i}}(v_{i}-a_i)\le v_{i_0}<\frac{d_{i_0}}{d_{i}}v_{i}+a_{i_0}.$$
Immediately, $\frac{d_{i_0}}{d_{i}}(v_{i}-a_i)\le v_{i_0}+a_{i_0}$. If $i\in \mathcal{I}_v$, then $\frac{v_{i_0}}{d_{i_0}}<\frac{v_i}{d_i}$ by the definition of $i_0$. So $v_{i_0}+a_{i_0}<\frac{d_i}{d_{i_0}}v_i+a_{i_0}$. If $i\not\in \mathcal{I}_v$, then from $v_i>0\ge v_{i_0}$ one also has $v_{i_0}+a_{i_0}<\frac{d_i}{d_{i_0}}v_i+a_{i_0}$. Thus we obtain
$$\frac{d_{i_0}}{d_{i}}(u_{i}-a_i)\le u_{i_0}<\frac{d_i}{d_{i_0}}u_i+a_{i_0}$$
for $1\le i<i_0$. Similarly, we can prove
$$\frac{d_{i}}{d_{i_0}}(u_{i_0}-a_{i_0})\le u_{i}<\frac{d_{i}}{d_{i_0}}u_{i_0}+a_{i}$$ for $i_0<i\le n-1$. Hence $u\in \mathcal{A}$ and $\psi(\mathcal{T})\subseteq \mathcal{S}$.

We now show that $\varphi(\mathcal{S})\subseteq \mathcal{T}$. Let $u'\in \mathcal{S}$.
Denote by $v':=\varphi(u')$ and write $v':=(v_1',...,v_n')$. Clearly,
\begin{equation}\label{eq2}
v_n'=u_n'+a_n<\frac{d_n}{d_{j_0}}(u_{j_0}'-a_{j_0})+a_{n}=\frac{d_n}{d_{j_0}}v_{j_0}'+a_{n}.
\end{equation}
Note that $-d_n<u_n'<0$. Hence $-d_{j_0}<v_{j_0}'\le 0$.
For $1\le i<j_0$, one has $$\frac{d_{j_0}}{d_i}(u_i'-a_i)\le u_{j_0}'< \frac{d_{j_0}}{d_i}u_i'+a_{j_0}.$$
It then follows from (\ref{eq2}) that
$$v_n'=u_n'+a_n<\frac{d_n}{d_i}u_i'+a_n=\frac{d_n}{d_i}v_i'+a_n $$
for $1\le i<j_0$.
For $j_0<i\le n-1$, one has that
 $$ \frac{d_i}{d_{j_0}}(u_{j_0}'-a_{j_0})\le u_{i}'< \frac{d_i}{d_{j_0}}u_{j_0}'+a_{i}.$$
 From (\ref{eq2}), one has
 $$v_n'=u_n'+a_n<\frac{d_n}{d_i}u_i'+a_n=\frac{d_n}{d_i}v_i'+a_n  $$
 for $j_0<i\le n-1$.
Hence to show $\varphi(\mathcal{S})\subseteq \mathcal{T}$, it remains to show that $v'\in \mathcal{A}$. Obviously, $-d_i<v_i'\le a_i$ for $i=1,...,n$.
If $j_0\neq i,j<n$, then $v_i'=u_i', v_j'=u_j'$.
Hence (\ref{eq1}) is true for $v_i',v_j'$.

For $1\le i<j_0$, one has
$$\frac{d_{j_0}}{d_{i}}(u_{i}'-a_i)\le u_{j_0}'<\frac{d_{j_0}}{d_{i}}u_{i}'+a_{j_0}.$$
Immediately, $u_{j_0}'-a_{j_0}<\frac{d_{j_0}}{d_{i}}u_{i}'+a_{j_0}$. If $i\in \mathcal{J}_{u'}$, then $\frac{u_{j_0}'-a_{j_0}}{d_{j_0}}\ge \frac{u_i'-a_i}{d_i}$ by the definition of $j_0$. So $u_{j_0}'-a_{j_0}\ge \frac{d_{j_0}}{d_{i}}(u_{i}'-a_i)$. If $i\not\in \mathcal{J}_{u'}$, it then follows from $u_n'\ge \frac{d_n}{d_i}(u_i'-a_i)$ that $u_{j_0}'-a_{j_0}\ge \frac{d_{j_0}}{d_{i}}(u_i'-a_{i})$. Thus
$$\frac{d_{j_0}}{d_{i}}(u_{i}'-a_i)\le u_{j_0}'-a_{j_0}<\frac{d_{j_0}}{d_{i}}u_i'+a_{j_0}$$
 for $1\le i<j_0$. By some similar arguments we can prove
 $$\frac{d_{i}}{d_{j_0}}(u_{j_0}'-2a_{j_0})\le u_{i}'<\frac{d_{i}}{d_{j_0}}(u_{j_0}'-a_{j_0})+a_{i}$$ for $j_0<i\le n-1$. That is, $v'\in \mathcal{A}$. Hence $\varphi(\mathcal{S})\subseteq \mathcal{T}$.

This finishes the proof of Lemma \ref{lem3}.
\end{proof}

 \begin{lem}\label{lem1}
$\varphi\circ \psi=Id|_{\mathcal{T}}$ and $\psi \circ \varphi=Id|_{\mathcal{S}}$.
\end{lem}

\begin{proof}
First, we show $\varphi\circ \psi=Id|_{\mathcal{T}}$.
Let $v\in \mathcal{T}$ and $i_0\in \mathcal{I}_v$ be the least integer such that
$\frac{v_{i_0}}{d_{i_0}}=\min\big\{\frac{v_i}{d_i}: i\in \mathcal{I}_v\big\}$.
Then $\psi(v)=u:=v+A_{i_0}-A_n\in \mathcal{S}$.
For $i<i_0$, one has
$$\frac{d_{i_0}}{d_{i}}(v_{i}-a_{i})\le v_{i_0}<\frac{d_{i_0}}{d_{i}}v_{i}+a_{i_0}. $$
Hence
$$ v_{i_0}+a_{i_0}-a_{i_0}\ge \frac{d_{i_0}}{d_{i}}(v_{i}-a_{i})$$
for $i<i_0$.
For $i>i_0$, one has
$$\frac{d_{i}}{d_{i_0}}(v_{i_0}-a_{i_0})\le v_{i}<\frac{d_{i}}{d_{i_0}}v_{i_0}+a_{i}. $$
Hence
$$ v_{i_0}+a_{i_0}-a_{i_0}>\frac{d_{i_0}}{d_{i}}(v_{i}-a_{i})$$
for $i>i_0$.
We conclude that $i_0$ is the largest element of $\mathcal{J}_{u}$ such that
$$\frac{u_{i_0}-a_{i_0}}{d_{i_0}}=\max \big\{ \frac{u_j-a_j}{d_j}: j\in \mathcal{J}_{u}\big\} .$$
It follows that $$\varphi\circ \psi(v)=\varphi(v+A_{i_0}-A_n)=v+A_{i_0}-A_n-A_{i_0}+A_n=v.$$
Thus $\varphi\circ \psi=Id|_{\mathcal{T}}$.

By the same argument, we can show that if $u\in \mathcal{S}$, then $\psi\circ \varphi (u)=u$.
That is, $\psi \circ \varphi=Id|_{\mathcal{S}}$.
This finishes the proof of Lemma \ref{lem1}.
\end{proof}

It follows that $|\mathcal{A}_0|=|\mathcal{M}|$, from which we have the following.

\begin{lem}\label{lem2}
For $n\ge 1$, $|\mathcal{B}|=n!Vol \Delta(\bar{F}).$
\end{lem}
\begin{proof}
Let $\mathcal{B}_1$ be the set $v\in \mathbb{Z}$ such that $-d_1<v\le a_1$.
For $n\ge 2$, we let $\mathcal{B}_n$ be the set of $(v_1,\cdots, v_n)\in \mathbb{Z}^n$ such that
$-d_i<v_i\le a_i$ for all $i=1,\cdots,n$, and
$v_i,v_j$ such that (\ref{eq1})
for $1\le i<j\le n$.

In what follows, we compute $|\mathcal{B}_n|$ by induction.
Let $n=1$. It is evidently that $|\mathcal{B}_1|=a_1+d_1.$
Let $n=2$. Then we count that $$|\mathcal{B}_2|=a_1a_2+a_1d_2+a_2d_1.$$
Consider $n\ge 3$. Assume that
$$|\mathcal{B}_{n-1}|=\prod_{i=1}^{n-1} a_i+\sum_{j=1}^{n-1}d_j\prod_{i=1,i\neq j}^{n-1}a_i.$$
The number of lattices in $\mathcal{B}_{n-1}$ satisfying $v_i> 0$ for all $i=1,...,n-1$ is $\prod_{i=1}^{n-1} a_i$.
Hence the number of lattices in $\mathcal{B}_{n-1}$ such that $-d_i<v_i\le 0$ for some $1\le i \le n-1$
 is $\sum_{j=1}^{n-1}d_j\prod_{i=1,i\neq j}^{n-1}a_i$.
 Hence
  $$|\mathcal{M}|=d_n\sum_{j=1}^{n-1}d_j\prod_{i=1,i\neq j}^{n-1}a_i.$$
Clearly,
$$|\mathcal{A}|=(\prod_{i=1}^{n-1} a_i+\sum_{j=1}^{n-1}d_j\prod_{i=1,i\neq j}^{n-1}a_i )(a_n+d_n).$$
Note that $\mathcal{B}_n=\mathcal{A}-\mathcal{A}_0$.
It then follows from $|\mathcal{M}|=|\mathcal{A}_0|$ that
$$|\mathcal{B}_n|=\Big(\prod_{i=1}^{n-1} a_i+\sum_{j=1}^{n-1}d_j\prod_{i=1,i\neq j}^{n-1}a_i \Big)(a_n+d_n)-|\mathcal{M}|=\prod_{i=1}^{n} a_i+\sum_{j=1}^{n}d_j\prod_{i=1,i\neq j}^{n}a_i.$$

Clearly, $|\mathcal{B}|=|\mathcal{B}_1|=a_1+d_1=\Delta(\bar{F})$ for $n=1$. For $n\ge 2$, we compute
$$n!Vol \Delta(\bar{F})=\prod_{i=1}^n a_i+\sum_{j=1}^nd_j\prod_{i=1,i\neq j}^na_i=|\mathcal{B}_n|=|\mathcal{B}|.$$
This finishes the proof of Lemma \ref{lem2}.
\end{proof}

Now we can give the proof of Theorem \ref{thm10}.

{\it Proof of Theorem \ref{thm10}.} It follows from Lemma \ref{thm3} and Lemma \ref{lem2} that $\bar{B}$ is a basis of
$$H^n(\Omega^{\bullet}(\bar{R},\nabla(\bar{F})))=\bar{R}/\sum_{l=1}^nx_l \frac{\partial \bar{F}}{\partial x_l}\bar{R} .$$
Then by Proposition \ref{prop1} one has that Theorem \ref{thm10} is true.
\hfill$\Box$

{\bf Remark:} It follows from Theorem \ref{thm10} that the Hodge polygon $HP(\Delta(\bar{F}))$
can be obtained by computing the weights of elements in set $\mathcal{B}$.
The following is how to compute the Hodge polygon.
Let $u\in \mathcal{B}$.
For $n=1$, then $w(u)=\frac{u}{a_1}$ if $0\le u\le a_1$, and $w(u)=\frac{-u}{d_1}$ if $-d_1< u<0$.
For $n\ge 2$,
if $u\in C(\Delta_{0})$, then let
$$w(u)=\sum_{i=1}^n\frac{u_i}{a_i}.$$
If $u\in C(\Delta_k)$, then we let
$$w(u)=\frac{u_k}{\frac{-d_k}{1+\frac{d_1}{a_1}+...\frac{d_n}{a_n}-\frac{d_k}{a_k}}}+\sum_{i=1}^n\frac{u_i}{a_i}-\frac{u_k}{a_k}.$$

Hence $e=lcm(a_1,d_1)$ for $n=1$, and $e=lcm(a_1,...,a_n)\cdot lcm(d_1,...,d_n)$ for $n\ge 2$.
Then $e=e^{*}$.
Denote $H_{\Delta}(k)=|\big\{u\in \mathcal{B}: w(u)=\frac{k}{e}\big\}|$.
The Hodge polygon is the lower convex polygon in $\mathbb{R}^2$ with vertices $(0,0)$ and
$$\Big(\sum_{k=0}^{m}H_{\Delta}(k),\sum_{k=0}^{m}\frac{k}{e}H_{\Delta}(k) \Big), m=0,1,...,ne.$$
Thus Corollary \ref{thm2} follows.

\subsection{Newton polygon}

\begin{thm}\label{thm8}(Facial decomposition theorem, \cite{WD1}) Let $f$ be a nondegenerate Laurent polynomial with $n$ variables over $\mathbb{F}_q$.
Assume $\Delta=\Delta(f)$ is $n$-dimensional and $\Delta_1,\cdots,\Delta_h$ are all the codimension 1 faces of $\Delta$
which do not contain the origin. Let $f^{\Delta_i}$ denote the restriction of $f$ to $\Delta_i$. Then $f$ is ordinary if and only if
$f^{\Delta_i}$ is ordinary for $1\le i\le h$.
\end{thm}

In what follows, we introduce some criteria to determine the nondegenerate and ordinary property.

A Laurent polynomial $f\in \mathbb{F}_q[x_1^{\pm},\cdots,x_n^{\pm}]$ is called {\it diagonal} if $f$ has exactly $n$ non-constant terms and $\Delta(f)$ is $n$-dimensional.
Let $f(x)=\sum_{j=1}^na_jx^{V_j}$ with $a_j\in \mathbb{F}_q^{*}$.
The square matrix of $\Delta$ is defined to be
$$\mathbf{M}(\Delta)=(V_1,\cdots,V_n),$$
where each $V_j$ is written as a column vector.
If $f$ is diagonal, then $\det\mathbf{M}(\Delta)\neq0$.

 Suppose $f\in \mathbb{F}_q[x_1^{\pm},\cdots,x_n^{\pm}]$ is diagonal with $\Delta(f)$. It is well-known
 that $f$ is nondegenerate if and only if
$\gcd(p,\det \mathbf{M}(\Delta))=1$.

Let $S(\Delta)$ be the solution set of the following linear system
$$\mathbf{M}(\Delta)\cdot(r_1,r_2,...,r_n)^t \equiv 0 \pmod 1, r_i\in \mathbb{Q}\cap [0,1),$$
where $(r_1,r_2,...,r_n)^t$ means transposition of $(r_1,r_2,...,r_n)$.
Then $S(\Delta)$ is an abelian group and its order is given by $|\det \mathbf{M}(\Delta)|$.
By the fundamental structure of finite abelian group, $S(\Delta)$ can be decomposed into a direct product of invariant factors
$$S(\Delta)=\bigoplus_{i=1}^n \mathbb{Z}/s_i(\Delta)\mathbb{Z},$$
where $s_i(\Delta)|s_{i+1}(\Delta)$ for $i=1,2,...,n-1$. Then Wan proved the following ordinary criterion.
\begin{prop}\label{prop2}\cite{WD2} Suppose $f\in \mathbb{F}_q[x_1^{\pm},\cdots,x_n^{\pm}]$ is a nondegenerate diagonal Laurent polynomial with
$\Delta(f)$. Let $s_n(\Delta)$ be the largest invariant factor of $S(\Delta)$. If $p\equiv 1 \mod s_n(\Delta)$, then $f$ is ordinary.
\end{prop}

Now we can use Wan's results to prove Theorem \ref{thm9}.

{\it Proof of Theorem \ref{thm9}.}
First we consider $n=1$. Then $\bar{F}^{\Delta_{0}}=x_1^{a_1}$ and $\bar{F}^{\Delta_{1}}=x_1^{-d_1}$.
Note that $p\nmid a_1d_1$. Hence $\bar{F}^{\Delta_0}$ and $\bar{F}^{\Delta_{1}}$
are nondegenerate. Clearly, $s_1(\Delta_0)=a_1$ and $s_1(\Delta_1)=d_1$.
From Theorem \ref{thm8} and Proposition \ref{prop2} one has that if $p\equiv 1\mod lcm(a_1,d_1)$, then $\bar{F}$ is ordinary.

In what follows, we let $n\ge 2$.
It is easy to see that $\Delta_0,\Delta_1,...,\Delta_{n}$ are all codimension 1 faces of $\Delta(\bar{F})$
which do not contain the origin.
Then let  $\bar{F}^{\Delta_j}=\sum_{i=1,i\neq j}^nx_i^{a_i}+\frac{\bar{\lambda}}{\prod_{i=1}^nx_i^{d_i}}$ for $j=1,...,n$
and $\bar{F}^{\Delta_{0}}=\sum_{i=1}^nx_i^{a_i}$.
We compute $|\det \mathbf{M}(\Delta_j)|=d_j\prod_{i=1,i\neq j}^na_i$ and $|\det \mathbf{M}(\Delta_{0})|=\prod_{i=1}^na_i$.
Note that $p\nmid \prod_{i=1}^na_id_i$. Hence $\bar{F}^{\Delta_0}$, ..., $\bar{F}^{\Delta_{n}}$
are nondegenerate.

We compute $s_n(\Delta_{0})=lcm(a_1,...,a_n)$ and $lcm(a_1,...,\hat{a}_j,..., a_n)d_j$ is divisible by $s_n(\Delta_j)$  for $j=1,...,n$, where $\hat{a}_j$
means that $a_j$ is omitted.
It follows from Proposition \ref{prop2} that if $$p\equiv 1 \mod d_jlcm(a_1,...,\hat{a}_j,..., a_n)$$ for $j=1,...,n$, and
 $p\equiv 1 \mod lcm(a_1,...,a_n)$, then
 $$NP(\bar{F}^{\Delta_j})=HP(\Delta_j)$$ for $j=0,1,...,n$.
It then follows from Theorem \ref{thm8} that if $$p\equiv 1 \mod lcm(a_1,...,a_n)\cdot lcm(d_1,...,d_n),$$  then
$$NP(\bar{F})=HP(\Delta(\bar{F})).$$
This finishes the proof of Theorem \ref{thm9}.
\hfill$\Box$

Hence under the condition of Theorem \ref{thm9}, the Newton polygon of the $L$-function $L(\bar{F}(\bar{\lambda},x),T)^{(-1)^{n-1}}$
can be computed by the Hodge polygon given by Corollary \ref{thm2}. Finally, we prove Theorem \ref{thm1}.

{\it Proof of Theorem \ref{thm1}.}
It follows from Theorem \ref{thm9} that $NP(\bar{F})=HP(\Delta(\bar{F}))$. Note that $d_i=1$ for all $i\in \{1,\cdots, n\}$.
Let $u\in \mathcal{B}$. For $n=1$, then $0\le u\le a_1$ and $w(u)=\frac{u}{a_1}$. Hence the slope sequence of $L(\bar{F}(\bar{\lambda},x),T)$ is the set
$$\big\{\frac{u}{a_1}: u= 0,\cdots, a_1\big\}.$$

For $n\ge 2$,
 Theorem \ref{thm10} tells us that the basis of $H^n(\mathcal{C}_{0}, \nabla (D))$ is
$$\mathcal{B}=\Big\{\prod_{i=1}^nx_i^{u_i} \Big\}_{0\le u_i\le a_i}-\Big\{\prod_{i=1}^nx_i^{v_i} \Big\}_{(v_i,v_j)=(0,a_j) {\rm\ for}\ 0\le i<j\le n}.$$
 Then $w(u)=\sum_{i=1}^n\frac{u_i}{a_i}$.
Suppose there exists $v\in \mathcal{B}$ such that $v\neq u$ and $w(v)=w(u)$, that is,
$$\sum_{i=1}^n\frac{u_i}{a_i}=\sum_{i=1}^n\frac{v_i}{a_i}.$$
Note that $\gcd(a_i,a_j)=1\ {\rm for\ any}\ 1\le i\neq j\le n$. Then we have
$a_i \mid (u_i-v_i)$ for all $i=1,\cdots,n$.
Note that $0\le u_i, v_i\le a_i$ for $i=1,...,n$. It follows that if
$v_i\neq 0, a_i$, then $u_i=v_i$. Hence $v_i=0, a_i$ for some $i=1,...,n$.
 Suppose that $i_1,...,i_k$ are the integers such that $v_{i_l}=0,a_{i_l}$ for $l=1,...,k$.
  Let $v'=(v_{i_1},...,v_{i_k})$. By the fact that $(v_i,v_j)\neq (0,a_j)$ for $1\le i< j\le n$, one has that $v'$ should be $$(0,...,0),(a_{i_1},0,...,0),\cdots,(a_{i_1},...,a_{i_k}),$$
   which have different weights. Hence if $w(u)=w(v)$, then $u=v$.
  Thus each element in $\mathcal{B}$ has different weight. It follows that the slope sequence of $L(\bar{F}(\bar{\lambda},x),T)^{(-1)^{n-1}}$ equals to
  $$\Big\{\sum_{i=1}^n\frac{u_i}{a_i}: u_i= 0,\cdots, a_i\Big\}$$ as set. This finishes the proof of Theorem \ref{thm1}.
\hfill$\Box$
\begin{center}
{\sc Acknowledgements}
\end{center}
We wish to thank Steven Sperber for suggesting the subject considered in this paper,
Daqing Wan for much helpful advice.
 Liping Yang would like to thank Liman Chen, Huaiqian Li and Hao Zhang for many
helpful discussions.

\bibliographystyle{amsplain}

\end{document}